\documentclass[]{article}
\usepackage{graphicx}
\usepackage{amsfonts}
\usepackage{amsmath}
\usepackage{amssymb}
\usepackage{fancyhdr}
\usepackage{titlesec}
\usepackage{indentfirst}
\usepackage{booktabs}
\usepackage{verbatim}
\usepackage{color}
\usepackage{amsthm}
\usepackage{subfigure}

\usepackage[page,header]{appendix}
\usepackage{titletoc}

\newcommand{\rd}{\,\mathrm{d}}
\numberwithin{equation}{section}
\newtheorem{theorem}{Theorem}[section]
\newtheorem{lemma}[theorem]{Lemma}

\newtheorem{definition}[theorem]{Definition}
\newtheorem{remark}[theorem]{Remark}

\def\bx{{\bf x}}
\def\by{{\bf y}}
\def\bz{{\bf z}}

\def\bv{{\bf v}}

\def\cA{\mathcal{A}}

\def\var{\textnormal{var}}

\begin{document}
\title{Collective dynamics of opposing groups with stochastic communication}
\author{Shi Jin\footnote{School of Mathematical Sciences, Institute of Natural Sciences, MOE-LSEC and SHL-MAC, Shanghai Jiao Tong University, Shanghai 200240, China. Email: shijin-m@sjtu.edu.cn. SJ was partially supported by NSFC
grants Nos. 31571071 and 11871297.} and Ruiwen Shu\footnote{University of Maryland, College Park, 4176 Campus Drive, College Park, MD 20742, USA. Email: rshu@cscamm.umd.edu. RS was supported in part by NSF grants DMS16-13911, RNMS11-07444 (KI-Net) and ONR grant N00014-1812465.}}
\maketitle
\begin{abstract}
We propose models describing the collective dynamics of two {\it opposing} groups of individuals with {\it stochastic} communication. Individuals from the same group are assumed to align in a stochastic manner, while individuals from different groups are assumed to anti-align. Under reasonable assumptions, we prove the large time behavior of {\it separation}, in the sense that the variation inside a group is much less than the distance between the two groups. The separation phenomena are verified by numerical simulations.
\end{abstract}

\section{Introduction}

Collective dynamics~\cite{BelHT,BCL,BHT,CS1,CS2,DM,DMPW,HHK,HT,HK,MT,Tos} is an important phenomenon in social and biological sciences. It describes the dynamics of a large group of agents resulted from their communications or interactions. For example, in opinion dynamics, for a group of $N$ people, the opinion of the $i$-th person at time $t$ is modeled by a vector $\bz_i(t)\in \mathbb{R}^d$, whose time evolution follows the first order dynamics
\begin{equation}\label{eq0}
\dot{\bz}_i = \frac{1}{N}\sum_{i'\ne i} \psi_{i,i'}(\bz_i,\bz_{i'})(\bz_{i'}-\bz_i),\quad i=1,\dots,N.
\end{equation} 
Here $\psi_{i,i'}(\bz_i,\bz_{i'})$ is the interaction kernel for a pair of individuals $i$ and $i'$, describing how their communication affects their opinions. One usually assumes 
\begin{equation}\label{assu0}
\psi_{i,i'}(\bz_i,\bz_{i'})=\psi(|\bz_i-\bz_{i'}|) \ge 0.
\end{equation}
This means (1)  people in the group are {\it indistinguishable}, i.e., the interaction law does not depend on their indices but their distances; (2) people tend to {\it align} their opinions with each other via communication. Under the assumption \eqref{assu0}, Motsch and Tadmor~\cite{MT} proved {\it consensus} of opinions $\lim_{t\rightarrow\infty} \max_{i,i'}|\bz_i-\bz_{i'}| = 0$ for lower bounded interaction kernel, and {\it formation of clusters} for compactly supported interaction kernel under certain assumptions on initial data.

However, there are realistic situations where the assumption \eqref{assu0} may not be true:
\begin{itemize}
\item The communication of different pairs of people may be drastically different. For example, people with similar backgrounds or within the same community are more likely to align their opinions, while people from opposing groups may tend to {\it anti-align} their opinions, i.e., make one's opinion farther from the others'.
\item Not every pair of people are necessarily communicating with each other all the time, and their chance of communication may be time dependent. In fact, the communication among people has intrinsic feature of {\it stochasticity}: the best one cay say is that a pair of people communicate with certain probability.
\end{itemize}
In this work, we propose an opinion dynamics model in the framework of \eqref{eq0} to capture both issues mentioned above. To model the anti-alignment between opposing groups, we consider two groups of people, with opinions $\{\bx_i(t)\}_{i=1}^{N_1}$ and $\{\by_j(t)\}_{j=1}^{N_2}$, satisfying the first order model\footnote{From now on, summations in $i,i'$ always run from 1 to $N_1$, and $j,j'$ from 1 to $N_2$. $\sum_{i'\ne i}$ means summation over the index $i'$ which is not equal to $i$. }
\begin{equation}\label{eq}\left\{\begin{split}
& \dot{\bx}_i = \frac{1}{N_1}\sum_{i'\ne i} \psi^+_{i,i'}(t) (\bx_{i'}-\bx_i) - \frac{1}{N_2}\sum_{j} \psi^-_{i,j}(t) (\by_j-\bx_i) ,\\
& \dot{\by}_j = \frac{1}{N_2}\sum_{j'\ne j} \psi^+_{j,j'}(t) (\by_{j'}-\by_j) - \frac{1}{N_1}\sum_{i} \psi^-_{j,i}(t) (\bx_i-\by_j) .\\
\end{split}\right.\end{equation}
Here $\psi^+_{i,i'}=\psi^+_{i',i}\ge 0$ and $\psi^+_{j,j'}=\psi^+_{j',j}\ge 0$ describe the alignment of opinions between people in the same group, and $\psi^-_{i,j}=\psi^-_{j,i}\ge 0$ describes the anti-alignment between people in opposing groups. 

We propose two scenarios to describe the stochasticity of communication. In the first scenario, we assume that whether a pair of people are communicating is random but do not change with time. To be precise, we assume $\psi^+_{i,i'}$ and $\psi^+_{j,j'}$ are constants, given randomly by the Bernoulli distribution
\begin{subequations}\label{rand1}
\begin{equation}\label{rand1_1}
\psi^+_{i,i'} \sim B(p) =  \left\{\begin{split}
& 1,\quad \textnormal{with probability $p$} \\
& 0,\quad \textnormal{with probability $1-p$} \\
\end{split}\right.,\quad \psi^+_{j,j'}  \sim B(p),
\end{equation}
for some fixed $0<p<1$, representing the rate of communication between people in the same group. We assume similarly for $\psi^-_{i,j}$:
\begin{equation}\label{rand1_2}
\psi^-_{i,j}  \sim B(q),
\end{equation}
\end{subequations}
for some fixed $0<q<1$, representing the rate of communication between people in opposing groups. The random variables $\psi^{+}_{i,i'}$, $\psi^{+}_{j,j'}$ and $\psi^{-}_{i,j}$ are assumed to be mutually independent for different indices $i,j,i',j'$.

In the second scenario, we assume that whether a pair of people are communicating is random, and change at time $k\tau,\,k\in\mathbb{Z}$, where $\tau>0$ is a fixed time step. To be precise  we assume $\psi^+_{i,i'}$ and $\psi^+_{j,j'}$ are functions of time, being constant on each interval $[k\tau,(k+1)\tau),\,k=0,1,\dots$, given randomly by
\begin{subequations}\label{rand2}
\begin{equation}\label{rand2_1}
\psi^+_{i,i'}|_{[k\tau,(k+1)\tau)} = \psi^{+,k}_{i,i'}   \sim B(p),\quad \psi^+_{j,j'}|_{[k\tau,(k+1)\tau)} =\psi^{+,k}_{j,j'}  \sim B(p),
\end{equation}
and similarly for $\psi^-_{i,j}$:
\begin{equation}\label{rand2_2}
  \psi^-_{i,j}|_{[k\tau,(k+1)\tau)} =\psi^{-,k}_{i,j}\sim B(q).
\end{equation}
\end{subequations}
In this model, after time $\tau$, the communication functions will be resampled, though by the same distribution.

Due to the alignment effect within a group and anti-alignment effect between different groups, it is natural to expect that the opinions of different groups will reach {\it separation}, i.e., the disagreement of opinions between any two people from different groups is much larger than that inside the same group. In our main results (Theorems \ref{thm1} and \ref{thm2}) we prove exactly such large time behavior of \eqref{eq} with large probability for both scenarios, under reasonable assumptions on parameters and initial data. This forms sharp contrast with the large time behavior of {\it consensus} proved in~\cite{MT} for global alignment interactions.

In this paper we assume that $\psi^+_{i,i'}$, $\psi^+_{j,j'}$ and $\psi^-_{i,j}$ are independent of $\bx_i$ and $\by_j$. This suits more the era of modern social media by which people can communicate easily regardless of the spatial distance between them. Although this assumption may seem restrictive, we believe that our main result, which is based on energy estimates, can be generalized to some cases with $\psi^+_{i,i'}$, $\psi^+_{j,j'}$ and $\psi^-_{i,j}$ depending on $\bx_i$ and $\by_j$. This is left as future work.

We remark that our model \eqref{eq} is closely related to the stochastic block model (SBM)~\cite{HLL}. We refer to~\cite{Abbe} and references therein for a thorough review of SBM, and some numerical experiments of a model we study in~\cite{JLL}. In its simplest version, SBM is a random graph, whose vertices are two groups with sizes $N_1$ and $N_2$ respectively, and the probability of the appearance of an edge depends on whether the two vertices are in the same group. Therefore the time evolution problem of its adjacency matrix is similar to \eqref{eq} with \eqref{rand1}, but instead of having anti-alignment, here the interaction between different groups is still alignment, but with a  communication rate smaller than that within the same group. In fact, as mentioned in~\cite{Abbe}, this evolution problem from SBM may not have a large time behavior of separation (i.e., the sign of components of the dominating eigenvector may not distinguish the two groups), and the main goal of studying the stochastic block model is to recognize the two groups (usually called {\it clustering}) using other methods. Although our model \eqref{eq} does not serve as a method for clustering,  its large time behavior is interesting for its own sake, from the perspective of opinion and other collective dynamics.

Our model \eqref{eq} with \eqref{rand2} can be viewed as a randomized version of a deterministic model, with coefficients $\psi^+_{i,i'}=\psi^+_{j,j'}=p$, $\psi^-_{i,j}=q$, in the same spirit as the deterministic bi-flocking model in~\cite{FHJ} and its applications to the data clustering problem~\cite{JLL}. Since it is straightforward to see the separation for this deterministic model, our separation result for \eqref{rand2} provides evidence that randomizing an interacting particle system may preserve its large time behavior with large probability, and therefore suggests that RBM could be effective to capture certain large time behavior of interacting particle systems. 

This paper is organized as follows: in Section 2 we introduce some basic notations and state our main results. In Section 3 we give some probabilistic descriptions of the coefficient matrices given by \eqref{rand1} and \eqref{rand2}, which are essential to the proof of the main results. In Section 4 we prove our main results. In Section 5 we give some numerical results verifying our main results. The paper is concluded in Section 6.

\section{Notations and main results}

We first give a precise description of the {\it separation} of locations (i.e., opinions of people) of two groups. Define the mean locations
\begin{equation}
\bar{\bx} = \frac{1}{N_1}\sum_i \bx_i,\quad \bar{\by} = \frac{1}{N_2}\sum_j \by_j,
\end{equation}
and the deviation of locations
\begin{equation}
\hat{\bx}_i = \bx_i-\bar{\bx},\quad \hat{\by}_j = \by_j-\bar{\by}.
\end{equation}
Define the variance of locations within each group:
\begin{equation}
\var(\bx) = \frac{1}{N_1}\sum_i|\hat{\bx}_i|^2,\quad \var(\by) = \frac{1}{N_2}\sum_j|\hat{\by}_j|^2.
\end{equation}
Then the relative size between $|\bar{\bx}-\bar{\by}|^2$ and $\var(\bx)+\var(\by)$ is an indicator of group separation in the sense of $L^2$: if $|\bar{\bx}-\bar{\by}|^2$ is much larger than $\var(\bx)+\var(\by)$, then the two groups are well-separated in an average sense.

We also concern with the separation in the $L^\infty$ sense, meaning that the two groups are completely separated by some hyperplane. If each $|\hat{\bx}_i|$ and $|\hat{\by}_j|$ is much smaller than $|\bar{\bx}-\bar{\by}|$, then the two group are completely separated by the perpendicular bisector of the segment connecting $\bar{\bx}$ and $\bar{\by}$.

\begin{definition}
We say an event $\cA_N$ depending a large parameter $N$ \emph{happens with large probability} if 
\begin{equation}
\lim_{N\rightarrow\infty} P(\cA_N) = 1.
\end{equation}
\end{definition}

We first state our separation result for the first scenario, i.e., \eqref{eq} with \eqref{rand1}.

\begin{theorem}\label{thm1}
Consider \eqref{eq} with parameters given by \eqref{rand1}. Assume $N_1$ and $N_2$ are comparable, in the sense that there exists constant $\kappa>0$ such that 
\begin{equation}\label{kappa}
\frac{1}{\kappa} N_1\le N_2 \le \kappa N_1.
\end{equation}
Then for sufficiently large $N= \min\{N_1,N_2\}$ (in terms of $p$, $q$ and $\kappa$), there exist constants $\lambda_-(N),\,\lambda_+(N)$, with
\begin{equation}\label{lambda}
\lambda_-(N) \le  \frac{C(p,q,\alpha)}{N^{1-\alpha}},\quad \lambda_+(N) \ge  c(p,q,\alpha)N^{1-\alpha},\quad \textnormal{for any }0<\alpha<1,
\end{equation}
such that, with large probability, if the initial data of \eqref{eq} satisfies
\begin{equation}\label{thm1_1}
\lambda(0):= \frac{\var(\bx(0)) + \var(\by(0))}{|\bar{\bx}(0)-\bar{\by}(0)|^2} \le \lambda_+,
\end{equation}
then there holds the $L^2$ separation
\begin{equation}\label{thm1_2}
\lambda(t)= \frac{\var(\bx(t)) + \var(\by(t))}{|\bar{\bx}(t)-\bar{\by}(t)|^2} \le \lambda_- + (\lambda(0)-\lambda_-) e^{-\mu t},\quad \mu=p/2,
\end{equation}
and $L^\infty$ separation
\begin{equation}\label{thm1_3}
\tilde{\lambda}(t) := \frac{\max_i |\hat{\bx}_i(t)|^2+\max_j |\hat{\by}_j(t)|^2}{|\bar{\bx}(t)-\bar{\by}(t)|^2} \le C(p,q,\kappa,\alpha)(\lambda_- + \tilde{\lambda}(0)e^{-\mu t}).
\end{equation}
\end{theorem}

The meaning of this theorem is as follows: for large $N$, with large probability, if initially the means of the two groups are not too close (compared to their variances), then for large time the two groups separate at exponential rate, in both $L^2$ and $L^\infty$ sense.

The proof of this theorem is based on $L^2$ estimates on $|\bar{\bx}-\bar{\by}|^2$ and $\var(\bx) + \var(\by)$, using certain good properties of the coefficient matrices which hold with large probability (see Section~\ref{sec_prob}). Then the $L^2$ separation is proved by an ODE stability argument (see Lemma \ref{lem_ode}). Using the $L^2$ result, we analyze the evolution of a single $\hat{\bx}_i$ and obtain the $L^\infty$ separation. 

Next we state our separation result for the second scenario, i.e., \eqref{eq} with \eqref{rand2}.

\begin{theorem}\label{thm2}
Consider \eqref{eq} with parameters given by \eqref{rand2}, and assume \eqref{kappa}. 
Then for any fixed $\Lambda>0$, if $N= \min\{N_1,N_2\}$ is sufficiently large (in terms of $\Lambda$, $\tau$, $p$, $q$ and $\kappa$), then with large probability, if the initial data of \eqref{eq} satisfies
\begin{equation}\label{thm2_1}
\tilde{\lambda}(0) := \frac{\max_i |\hat{\bx}_i(0)|^2+\max_j |\hat{\by}_j(0)|^2}{|\bar{\bx}(0)-\bar{\by}(0)|^2} \le \Lambda,
\end{equation}
then there holds the $L^\infty$ separation at some large time $T=T(\Lambda,\tau,p,q,\kappa)$, in the sense that there exists a constant unit vector $\bv$ such that 
\begin{equation}\label{thm2_2}
\max_i \Big(\bx_i(t)\cdot \bv \Big) < \min_j \Big(\by_j(t)\cdot \bv \Big) ,
\end{equation}
for all $t\ge T$.
\end{theorem}

This theorem gives the $L^\infty$ separation under slightly stronger assumption \eqref{thm2_1} on the initial data, compared to \eqref{thm1_1}. To prove this theorem, we basically apply Theorem \ref{thm1} to \eqref{eq} with \eqref{rand2} on the time interval $[0,K\tau]$ for some fixed large $K$ (the legitimacy of this application will be justified in the proof). Then the $L^\infty$ separation result \eqref{thm1_3} will imply \eqref{thm2_2} at time $T=K\tau$. Then a simple geometric argument shows that \eqref{thm2_2} will hold for all $t>T$ once it holds for $t=T$.

With the parameters given by \eqref{rand2}, we do not expect quantitative separation results like \eqref{thm1_2} or \eqref{thm1_3} to hold for all large time. In fact, there is always a positive probability such that the `good properties' of coefficient matrices do not hold on a long time interval $[K_1\tau,K_2\tau]$, with $(K_2-K_1)$ being arbitrarily large. Quantitative estimates like \eqref{thm1_2} or \eqref{thm1_3} may break down on such time intervals since they originate from a bi-stable ODE argument (see Lemma \ref{lem_ode}).

\section{Probabilistic descriptions of coefficient matrices}\label{sec_prob}

Before proceeding to the proof, we first need some good properties of typical coefficient matrices given randomly by Bernoulli distributions. For simplicity of notation, we will consider the case of \eqref{rand1}, and clearly the same results apply to the case of \eqref{rand2} for each fixed $k$.

\subsection{Description of $\{\psi^+_{i,i'}\}$ and $\{\psi^+_{j,j'}\}$}
For the alignment coefficient matrices $\{\psi^+_{i,i'}\}$ and $\{\psi^+_{j,j'}\}$, we need to control the $L^2$ contraction rate of the attraction part. We define the {\it Fiedler number} $\texttt{F}_1(\psi^{+})$ of $\{\psi^{+}_{i,i'}\}$ as the second smallest eigenvalue of the matrix $A=(a_{i,i'})$ given as\footnote{Notice that 0 is the smallest eigenvalue of $A$ with eigenvector $(1,\dots,1)^T$.}
\begin{equation}\label{A}
a_{i,i'}=\left\{\begin{split}
& -\frac{1}{N_1}\psi^{+}_{i,i'},\quad i\ne i', \\
& \frac{1}{N_1}\sum_{i''\ne i} \psi^{+}_{i,i''},\quad i=i',
\end{split}\right.
\end{equation}
and similarly define $\texttt{F}_2(\psi^{+})$ for $\{\psi^{+}_{j,j'}\}$.

The following lemma of Juhasz~\cite{Juh} says that the Fiedler numbers of $\{\psi^+_{i,i'}\}$ and $\{\psi^+_{j,j'}\}$ are close to $p$ with large probability, for large $N$.
\begin{lemma}[Theorem 2 of~\cite{Juh}]\label{lem_psi1}
Let $\{\psi^{+}_{i,i'}\}$ and $\{\psi^{+}_{j,j'}\}$ be given randomly by \eqref{rand1_1}. For any $\delta>0$ and $\epsilon>0$, there exists $N_0=N_0(p,\delta,\epsilon)$, such that for any $N= \min\{N_1,N_2\}>N_0$,
\begin{equation}
P\Big(|\texttt{F}_1(\psi^{+})-p|>\delta \textnormal{ or } |\texttt{F}_2(\psi^{+})-p|>\delta\Big) < \epsilon.
\end{equation}
\end{lemma}

\subsection{Description of $\{\psi^-_{i,j}\}$}
For the anti-alignment coefficient matrix $\{\psi^-_{i,j}\}$, we first introduce the following notations: Define the 'row (column) mean' of $\{\psi^{-}_{i,j}\}$
\begin{equation}
\Psi^{-}_i:=\frac{1}{N_2}\sum_j \psi^{-}_{i,j},\quad \Psi^{-}_j := \frac{1}{N_1}\sum_i \psi^{-}_{i,j},
\end{equation}
and the overall mean
\begin{equation}
\bar{\Psi}^{-} := \frac{1}{N_1N_2} \sum_{i,j} \psi^{-}_{i,j} = \frac{1}{N_1}\sum_i \Psi^{-}_i= \frac{1}{N_2}\sum_j \Psi^{-}_j.
\end{equation}

Define the 'maximal row (column) deviation'
\begin{equation}
\hat{\Psi}^{-}_i = \Psi^{-}_i - \bar{\Psi}^{-},\quad \hat{\Psi}^{-}_j = \Psi^{-}_j - \bar{\Psi}^{-},\quad D(\Psi^{-}) = \max\{\max_i |\hat{\Psi}^{-}_i|, \max_j |\hat{\Psi}^{-}_j|\}.
\end{equation}

Then we have the following lemma, which says that the row (column) mean is almost constant with large probability, for large $N$.

\begin{lemma}\label{lem_psi2}
 Assume \eqref{kappa} is satisfied for some $\kappa>0$. Let $\{\psi^{-}_{i,j}\}$ be given randomly by \eqref{rand1_2}. For any $0<\alpha<1$ and $\epsilon>0$, there exists $N_0=N_0(q,\kappa,\alpha,\epsilon)$, such that for any $N=\min\{N_1,N_2\}>N_0$,
\begin{equation}
P\left(|\Psi^{-}_i-q| \ge \frac{1}{N^{(1-\alpha)/2}} \textnormal{ for some }i\right) < \epsilon.
\end{equation}
\end{lemma}

This estimate is expected in view of the central limit theorem, since $\Psi^-_i$, as an average of $\{\psi^{-}_{i,j}\}_{j=1}^{N_2}$, approximately behaves like a normal distribution with mean $q$ and variance $O(1/N_2)$, and thus a deviation larger than $O(1/\sqrt{N_2})$ should have small probability.

\begin{proof}[Proof of Lemma \ref{lem_psi2}]

For a fixed $i$, $\{\psi^{-}_{i,j}\}_{j=1}^{N_2}$ are a set of $N_2$ i.i.d. random variables with Bernoulli distribution $B(q)$. Therefore for any $z>0$ with $zN_2$ being an integer,
\begin{equation}\begin{split}
P(\Psi^-_i \ge z) \le \sum_{n=zN_2}^{N_2} a_n,\quad a_n:={N_2 \choose n}q^n(1-q)^{N_2-n}.
\end{split}\end{equation}
Notice that
\begin{equation}
\frac{a_{n+1}}{a_n} = \frac{(N_2-n)q}{(n+1)(1-q)} \le 1,\quad \text{ if } n\ge qN_2,
\end{equation}
therefore for $z\ge q$, using the estimate $n^ne^{-n+1} \le n! \le (n+1)^{n+1}e^{-n}$,
\begin{equation}\begin{split}
P(\Psi^-_i \ge z) \le & N_2 a_{zN_2} = \frac{N_2\cdot N_2!}{(zN_2)!((1-z)N_2)!}q^{zN_2}(1-q)^{(1-z)N_2}\\
\le & \frac{N_2\cdot (N_2+1)^{N_2+1}e^{-N_2}}{(zN_2)^{zN_2}e^{-zN_2+1}\cdot ((1-z)N_2)^{(1-z)N_2}e^{-(1-z)N_2+1}}q^{zN_2}(1-q)^{(1-z)N_2} \\
= & \frac{N_2(N_2+1)}{e^2} \left[\Big(\frac{(N_2+1)q}{zN_2}\Big)^z\Big(\frac{(N_2+1)(1-q)}{(1-z)N_2}\Big)^{1-z}\right]^{N_2} \\
\le & \frac{N_2(N_2+1)}{e} \left[\Big(\frac{q}{z}\Big)^z\Big(\frac{1-q}{1-z}\Big)^{1-z}\right]^{N_2} \\
= & \frac{N_2(N_2+1)}{e} \exp\Big[N_2\Big(z(\log q-\log z) + (1-z)(\log (1-q)-\log (1-z))\Big)\Big] .\\
\end{split}\end{equation}
Notice that by writing $z = q+\delta$ and using Taylor expansion for small $\delta$,
\begin{equation}\begin{split}
& z(\log q-\log z) + (1-z)(\log (1-q)-\log (1-z)) \\
= & (q+\delta)\Big(\frac{-\delta}{q} + \frac{\delta^2}{2q^2}\Big) + (1-q-\delta)\Big(\frac{\delta}{1-q} + \frac{\delta^2}{2(1-q)^2}\Big) + O_q(\delta^3)
= -\frac{\delta^2}{2q(1-q)} + O_q(\delta^3).
\end{split}\end{equation}
For $0<\alpha<1$ and $N_2$ large, we can take
\begin{equation}
\delta_1 = \frac{1}{N_2^{(1-\alpha)/2}},\quad \delta = \delta_1 - \frac{1}{N_2}\Big\{(q+\delta_1)N_2\Big\} \le \delta_1,
\end{equation}
where $\{\cdot\}$ means taking the decimal part, and guarantee that $zN_2$ is an integer, and $\delta \ge \frac{9}{10}\delta_1$ for large $N_2$. Then we conclude
\begin{equation}\begin{split}
P\left(\Psi^-_i \ge q + \frac{1}{N_2^{(1-\alpha)/2}}\right) \le P\left(\Psi^-_i \ge q + \delta\right) \le  \frac{N_2(N_2+1)}{e}  \exp\Big(-\frac{N_2^\alpha}{4q(1-q)} \Big),\\
\end{split}\end{equation}
for large $N_2$. By similar argument for the event $\Psi^-_i \ge q - \frac{1}{N_2^{(1-\alpha)/2}}$, we finally obtain
\begin{equation}\begin{split}
P\left(|\Psi^-_i - q| \ge \frac{1}{N_2^{(1-\alpha)/2}}\right) \le  N_2^2  \exp\Big(-\frac{N_2^\alpha}{4q(1-q)} \Big),\\
\end{split}\end{equation}
for large $N_2$.

Since $\Psi^{-}_i$ are independent for different $i$, we have
\begin{equation}\begin{split}
& P\left( |\Psi^{-}_i-q| \ge \frac{1}{N^{(1-\alpha)/2}}\textnormal{ for some }i\in \{1,\dots,N_1\}\right) \\
= & 1 - \prod_{i=1}^{N_1}P\left( |\Psi^{-}_i-q| < \frac{1}{N^{(1-\alpha)/2}}\right) \\
\le &  1-\left(1-N_2^2  \exp\Big(-\frac{N_2^\alpha}{4q(1-q)} \Big)\right)^{N_1} \\
 \le & N_1N_2^2  \exp\Big(-\frac{N_2^\alpha}{4q(1-q)} \Big) \le \kappa N_2^3 \exp\Big(-\frac{N_2^\alpha}{4q(1-q)} \Big),
\end{split}\end{equation}
and the conclusion follows since the last quantity converges to 0 as $N=\min\{N_1,N_2\}\rightarrow \infty$.

\end{proof}

\begin{remark}
We can define similar notations for $\psi^{+}_{i,i'}$:
\begin{equation}
\Psi^{+}_i := \frac{1}{N_1}\sum_{i'\ne i} \psi^{+}_{i,i'},\quad \bar{\Psi}^{+} = \frac{1}{N_1}\sum_i \Psi^{+}_i,
\end{equation}
and
\begin{equation}
\hat{\Psi}^{+}_i = \Psi^{+}_i - \bar{\Psi}^{+},\quad D(\Psi^{+}) = \max_i |\hat{\Psi}^{+}_i|,
\end{equation}
and then the application of Lemma \ref{lem_psi2} to $\{\psi^{+}_{i,i'}\}$ with \eqref{rand2} gives 
\begin{equation}
P\left(|\Psi^{+}_i-p| > \frac{1}{N^{(1-\alpha)/2}} \textnormal{ for some }i\right) < \epsilon,
\end{equation}
for $N> N_0(p,\alpha,\epsilon)$.
\end{remark}

\section{Proof of Theorems \ref{thm1} and \ref{thm2}}

In this section we prove Theorems \ref{thm1} and \ref{thm2}. We first prove a lemma which will be used in the proof of Theorem \ref{thm1}.

\begin{lemma}\label{lem_ode}
Let $f(t),g(t)$ be positive functions satisfying the inequalities
\begin{equation}\begin{split}
& \dot{f} \ge A_{11}f - A_{12}g, \\
& \dot{g} \le A_{21}f - A_{22}g, \\
\end{split}\end{equation}
with $A_{12},A_{21}>0$. If the coefficients satisfy
\begin{equation}\label{lem1_1}
\Delta := (A_{11}+A_{22})^2-4A_{21}A_{12}>0,
\end{equation}
and initially 
\begin{equation}
\frac{g(0)}{f(0)} <\lambda_+,\quad \lambda_+:=  \frac{(A_{11}+A_{22}) +  \sqrt{\Delta}}{2A_{12}} ,
\end{equation}
then
\begin{equation}
\frac{g(t)}{f(t)} \le \lambda_- + \frac{g(0)}{f(0)}e^{-\mu t},\quad \lambda_-:= \frac{2A_{21}}{(A_{11}+A_{22}) + \sqrt{\Delta}},\quad \mu:=A_{12}(\lambda_+-\frac{g(0)}{f(0)}).
\end{equation}
\end{lemma}
\begin{proof}
Let $\lambda(t) =g(t)/f(t)$, and it satisfies
\begin{equation}
\dot{\lambda } = \frac{1}{f}\dot{g} - \frac{g}{f^2}\dot{f} \le A_{12}\lambda ^2 - (A_{11}+A_{22}) \lambda  + A_{21} = A_{12}(\lambda -\lambda_-)(\lambda -\lambda_+).
\end{equation}
If $\lambda (0)<\lambda_+$ as assumed, then $\dot{\lambda }<0$ as long as $\lambda_-<\lambda (t)<\lambda_+$, and thus $\lambda (t)$ cannot go above $\lambda (0)$. Therefore the conclusion follows from a comparison with the linear ODE $\dot{\lambda } = -\mu(\lambda -\lambda_-)$.
\end{proof}

\begin{proof}[Proof of Theorem \ref{thm1}]

Fix $\alpha>0$ small. Lemmas \ref{lem_psi1} and \ref{lem_psi2} imply that, if $N$ is large, then with large probability, there holds
\begin{equation}\label{psi_cond}\begin{split}
& \texttt{F}(\psi^{+}) := \min\{\texttt{F}_1(\psi^{+}), \texttt{F}_2(\psi^{+}) \} \ge p-\frac{p}{12}, \\
& |\bar{\Psi}^{-} - q| \le \min\left\{\frac{q}{24},\frac{p}{24}\right\}, \quad D(\Psi^{-}) \le \min\left\{\frac{1}{N^{(1-\alpha)/2}}, \frac{p}{24}\right\} ,\\
& |\bar{\Psi}^{+} - p| \le \frac{p}{24}, \quad D(\Psi^{+}) \le  \min\left\{\frac{1}{N^{(1-\alpha)/2}}, \frac{p}{24}\right\}. \\
\end{split}\end{equation}
Therefore it suffices to prove the large time behavior \eqref{thm1_2} and \eqref{thm1_3} for coefficient matrices satisfying \eqref{psi_cond}.


{\bf STEP 1}: $L^2$ estimate for $|\bar{\bx}-\bar{\by}|^2$.

The time evolution of $\bar{\bx}$ is given by
\begin{equation}\label{dbarx}\begin{split}
\dot{\bar{\bx}} = &  - \frac{1}{N_1N_2}\sum_{i,j} \psi^-_{i,j} (\by_j-\bx_i) = -\frac{1}{N_2}\sum_j \Psi^-_j \by_j + \frac{1}{N_1}\sum_i \Psi^-_i \bx_i \\ 
= & \bar{\Psi}^-(\bar{\bx}-\bar{\by}) + \frac{1}{N_1}\sum_i \hat{\Psi}^-_i \hat{\bx}_i-\frac{1}{N_2}\sum_j \hat{\Psi}^-_j \hat{\by}_j,
\end{split}\end{equation}
where the symmetry of $\psi^+_{i,i'}$ is used in the first equality. Therefore by subtracting its counterpart for $\bar{\by}$ and conducting energy estimate,
\begin{equation}\label{exp1}\begin{split}
\frac{1}{2}\frac{\rd}{\rd{t}}|\bar{\bx}-\bar{\by}|^2 = &  2 \bar{\Psi}^-|\bar{\bx}-\bar{\by}|^2 + 2(\bar{\bx}-\bar{\by})\cdot\left(\frac{1}{N_1}\sum_i \hat{\Psi}^-_i \hat{\bx}_i-\frac{1}{N_2}\sum_j \hat{\Psi}^-_j \hat{\by}_j\right)\\
\ge &  (2\bar{\Psi}^- - c_1)|\bar{\bx}-\bar{\by}|^2 - \frac{1}{c_1}\left|\frac{1}{N_1}\sum_i \hat{\Psi}^-_i \hat{\bx}_i-\frac{1}{N_2}\sum_j \hat{\Psi}^-_j \hat{\by}_j\right|^2\\
\ge &  (2\bar{\Psi}^- - c_1)|\bar{\bx}-\bar{\by}|^2 - \frac{2D(\Psi^-)^2}{c_1}\left(\frac{1}{N_1}\sum_i  |\hat{\bx}_i|^2+\frac{1}{N_2}\sum_j  |\hat{\by}_j|^2\right)\\
= &  (2\bar{\Psi}^- - c_1)|\bar{\bx}-\bar{\by}|^2 - \frac{2D(\Psi^-)^2}{c_1}(\var(\bx)+\var(\by)),\\
\end{split}\end{equation}
where we used Cauchy-Schwarz in the first inequality, with $c_1>0$ to be chosen. This gives an exponential growth estimate for $|\bar{\bx}-\bar{\by}|^2$, up to an error term of $(\var(\bx)+\var(\by))$ with small coefficient.

{\bf STEP 2}: $L^2$ estimate for $(\var(\bx)+\var(\by))$.

Subtracting \eqref{dbarx} from the first equation of \eqref{eq} gives
\begin{equation}\label{dhatx}\begin{split}
\dot{\hat{\bx}}_i = & \frac{1}{N_1}\sum_{i'\ne i} \psi^+_{i,i'} (\bx_{i'}-\bx_i) - \frac{1}{N_2}\sum_{j} \psi^-_{i,j} (\by_j-\bx_i) -\dot{\bar{\bx}} \\
= & \frac{1}{N_1}\sum_{i'\ne i} \psi^+_{i,i'} (\hat{\bx}_{i'}-\hat{\bx}_i) - \frac{1}{N_2}\sum_{j} \psi^-_{i,j} (\hat{\by}_j-\hat{\bx}_i) - \Psi^-_i (\bar{\by}-\bar{\bx}) -\dot{\bar{\bx}}. \\
\end{split}\end{equation}
Then 
\begin{equation}\label{dvar}\begin{split}
\frac{\rd}{\rd{t}}\left(\frac{1}{2N_1}\sum_i|\hat{\bx}_i|^2\right) = & \frac{1}{N_1^2}\sum_{i,i': i'\ne i} \psi^+_{i,i'} (\hat{\bx}_{i'}-\hat{\bx}_i)\cdot \hat{\bx}_i - \frac{1}{N_1N_2}\sum_{i,j} \psi^-_{i,j} (\hat{\by}_j-\hat{\bx}_i)\cdot \hat{\bx}_i \\ & - \frac{1}{N_1}\sum_i \hat{\Psi}^-_i (\bar{\by}-\bar{\bx})\cdot \hat{\bx}_i  - (\bar{\Psi}^- (\bar{\by}-\bar{\bx}) +\dot{\bar{\bx}})\cdot \frac{1}{N_1}\sum_i \hat{\bx}_i\\
= & -\frac{1}{2N_1^2}\sum_{i,i': i'\ne i} \psi^+_{i,i'} |\hat{\bx}_{i'}-\hat{\bx}_i|^2 - \frac{1}{N_1N_2}\sum_{i,j} \psi^-_{i,j} (\hat{\by}_j-\hat{\bx}_i)\cdot \hat{\bx}_i \\
& - (\bar{\by}-\bar{\bx})\cdot\frac{1}{N_1}\sum_i \hat{\Psi}^-_i  \hat{\bx}_i  \\
= & -I + II + III,
\end{split}\end{equation}
where we used $\sum_i \hat{\bx}_i=0$.

Notice that $I$ can be written as
\begin{equation}\label{est1}\begin{split}
\frac{1}{2N_1^2}\sum_{i,i': i'\ne i} \psi^+_{i,i'} |\hat{\bx}_{i'}-\hat{\bx}_i|^2 = &  \frac{1}{N_1^2}\sum_{k=1}^d \left(\sum_i\Big(\sum_{i'\ne i}\psi^+_{i,i'}\Big) |\hat{x}_i^{(k)}|^2 - \sum_{i,i': i'\ne i} \psi^+_{i,i'} \hat{x}_{i'}^{(k)}\hat{x}_i^{(k)}\right) \\
= &  \frac{1}{N_1}\sum_{k=1}^d  (\hat{x}^{(k)})^T A \hat{x}^{(k)},\\
\end{split}\end{equation}
where $\hat{x}^{(k)} = (\hat{x}_1^{(k)},\dots,\hat{x}_{N_1}^{(k)})^T$ with $\hat{x}_i^{(k)}$ denoting the $k$-th component of $\hat{\bx}_i$, and $A$ is as given in \eqref{A}. Thus $I$ is bounded below by
\begin{equation}\label{good}
\frac{1}{2N_1^2}\sum_{i,i': i'\ne i} \psi^+_{i,i'} |\hat{\bx}_{i'}-\hat{\bx}_i|^2 \ge \texttt{F}_1(\psi^+)\frac{1}{N_1}\sum_i |\hat{\bx}_i|^2,
\end{equation}
since the vector $\hat{x}^{(k)}$ is orthogonal to the eigenvector $(1,\dots,1)^T$ of the symmetric matrix $A$ with the smallest eigenvalue 0.

$II$ is controlled by
\begin{equation}\label{est2}\begin{split}
& - \frac{1}{N_1N_2}\sum_{i,j} \psi^-_{i,j} (\hat{\by}_j-\hat{\bx}_i)\cdot \hat{\bx}_i \\
= & - \frac{1}{N_1N_2}\sum_{i,j} \psi^-_{i,j} \hat{\by}_j\cdot \hat{\bx}_i + \frac{1}{N_1}\sum_i\Psi^-_i|\hat{\bx}_i|^2 \\
\le &  \frac{1}{N_1N_2}\sum_{i,j} \psi^-_{i,j}(\frac{1}{2}|\hat{\by}_j|^2 + \frac{1}{2}|\hat{\bx}_i|^2) + (\bar{\Psi}^- + D(\Psi^-))\frac{1}{N_1}\sum_i|\hat{\bx}_i|^2 \\
= &  \frac{1}{N_2}\sum_{j} \Psi^-_{j}\frac{1}{2}|\hat{\by}_j|^2 + \frac{1}{N_1}\sum_{i} \Psi^-_{i}\frac{1}{2}|\hat{\bx}_i|^2 + (\bar{\Psi}^- + D(\Psi^-))\frac{1}{N_1}\sum_i|\hat{\bx}_i|^2 \\
= & (\bar{\Psi}^- + D(\Psi^-)) \left( \frac{3}{2}\frac{1}{N_1}\sum_i|\hat{\bx}_i|^2 + \frac{1}{2}\frac{1}{N_2}\sum_j|\hat{\by}_j|^2\right).
\end{split}\end{equation}

$III$ is controlled by
\begin{equation}\label{est3}\begin{split}
 - (\bar{\by}-\bar{\bx})\cdot\frac{1}{N_1}\sum_i \hat{\Psi}^-_i  \hat{\bx}_i \le & c|\bar{\bx}-\bar{\by}|^2 + \frac{1}{4c} \frac{1}{N_1}\sum_i |\hat{\Psi}^-_i|^2 |\hat{\bx}_i|^2 \le c|\bar{\bx}-\bar{\by}|^2 + \frac{ D(\Psi^-)^2}{4c} \frac{1}{N_1}\sum_i  |\hat{\bx}_i|^2,  \\
\end{split}\end{equation}
where $c>0$ is a constant to be chosen.

Using \eqref{est1}, \eqref{est2} and \eqref{est3} in \eqref{dvar}, we obtain
\begin{equation}\begin{split}
& \frac{\rd}{\rd{t}}\left(\frac{1}{2N_1}\sum_i|\hat{\bx}_i|^2\right)\\
 \le & -\texttt{F}_1(\psi^+)\frac{1}{N_1}\sum_i |\hat{\bx}_i|^2 + (\bar{\Psi}^- + D(\Psi^-)) \left( \frac{3}{2}\frac{1}{N_1}\sum_i|\hat{\bx}_i|^2 + \frac{1}{2}\frac{1}{N_2}\sum_j|\hat{\by}_j|^2\right) \\
& + c|\bar{\bx}-\bar{\by}|^2 + \frac{D(\Psi^-)^2}{4c} \frac{1}{N_1}\sum_i  |\hat{\bx}_i|^2 \\
\le & -\texttt{F}_1(\psi^+)\frac{1}{N_1}\sum_i |\hat{\bx}_i|^2 + (\bar{\Psi}^- +\frac{p}{6}+ D(\Psi^-)) \left( \frac{3}{2}\frac{1}{N_1}\sum_i|\hat{\bx}_i|^2 + \frac{1}{2}\frac{1}{N_2}\sum_j|\hat{\by}_j|^2\right) \\
& + \frac{D(\Psi^-)^2}{p}|\bar{\bx}-\bar{\by}|^2,
 \end{split}\end{equation}
by choose $c=D(\Psi^-)^2/p$. Similarly,
\begin{equation}\begin{split}
& \frac{\rd}{\rd{t}}\left(\frac{1}{2N_2}\sum_j|\hat{\by}_j|^2\right) \\
\le & -\texttt{F}_2(\psi^+)\frac{1}{N_2}\sum_j |\hat{\by}_j|^2 + (\bar{\Psi}^- +\frac{p}{6} + D(\Psi^-)) \left( \frac{1}{2}\frac{1}{N_1}\sum_i|\hat{\bx}_i|^2 + \frac{3}{2}\frac{1}{N_2}\sum_j|\hat{\by}_j|^2\right) \\
& + \frac{D(\Psi^-)^2}{p}|\bar{\bx}-\bar{\by}|^2 .
\end{split}\end{equation}
Summing them together, we obtain
\begin{equation}\label{exp2}\begin{split}
& \frac{1}{2}\frac{\rd}{\rd{t}}\left(\var(\bx)+\var(\by) \right) 
\le  -\left(\texttt{F}(\psi^+) - 2(\bar{\Psi}^-  +\frac{p}{6} + D(\Psi^-)) \right) \left( \var(\bx)+\var(\by)\right) + \frac{2D(\Psi^-)^2}{p}|\bar{\bx}-\bar{\by}|^2.
\end{split}\end{equation}
This gives an exponential decay estimate for $(\var(\bx)+\var(\by))$, up to an error term of $|\bar{\bx}-\bar{\by}|^2$ with small coefficient.

{\bf STEP 3}: ODE stability analysis to get $L^2$ separation.

Now we claim that \eqref{exp1} and \eqref{exp2}, with \eqref{psi_cond}, imply the $L^2$ separation result \eqref{thm1_2}. In fact, we apply Lemma \ref{lem_ode} with
\begin{equation}
f = |\bar{\bx}-\bar{\by}|^2,\quad g = \var(\bx)+\var(\by),
\end{equation}
and
\begin{equation}
A_{11} = 2(2\bar{\Psi}-c_1),\,\, A_{12} = \frac{4D(\Psi^-)^2}{c_1},\,\, A_{21} = \frac{4D(\Psi^-)^2}{p},\,\, A_{22} = 2\left(\texttt{F}(\psi^+) - 2(\bar{\Psi}^-  +\frac{p}{6} + D(\Psi^-)) \right).
\end{equation}
Using \eqref{psi_cond}, one checks that with the choice $c_1=p/24$,
\begin{equation}\begin{split}
\frac{\Delta}{4} = & \frac{1}{4}(A_{11}+A_{22})^2-A_{21}A_{12} 
=  \left(\texttt{F}(\psi^+) - c_1  -\frac{p}{3} - 2D(\Psi^-) \right)^2 - \frac{16D(\Psi^-)^3}{pc_1} \\
\ge & \left(p -\frac{p}{24}-\frac{p}{24}-\frac{p}{3}-\frac{p}{12} \right)^2 - \frac{1000}{p^2 N^{3(1-\alpha)/2}} 
\ge  \frac{p^2}{4} - \frac{1000}{p^2 N^{3(1-\alpha)/2}}, \\
\end{split}\end{equation}
which is clearly positive if $N$ is large enough. Therefore Lemma \ref{lem_ode} applies, and gives the large time behavior \eqref{thm1_2} under the initial assumption \eqref{thm1_1}, with $\lambda_{\pm}$ and $\mu$ given as in the lemma. 

Then we estimate the asymptotic behavior of $\lambda_{\pm}$ and $\mu$ for large $N$. In fact, 
\begin{equation}
(A_{11}+A_{22}) + \sqrt{\Delta} \ge A_{11}+A_{22} \ge p,
\end{equation}
\begin{equation}
A_{12},\,A_{21} = O\left(\frac{1}{N^{1-\alpha}}\right),
\end{equation}
for large $N$. This gives the $L^2$ asymptotic behavior \eqref{lambda}. The fact that the decay rate $\mu=A_{12}(\lambda_+-\frac{g(0)}{f(0)})$ is independent of $N$ follows by further requiring $\frac{g(0)}{f(0)}<\lambda_+/2$ (with $\lambda_+$ as given by Lemma \ref{lem_ode}) and then 
\begin{equation}
\mu \ge A_{12}\lambda_+/2 = \Big((A_{11}+A_{22}) + \sqrt{\Delta}\Big)/2 \ge p/2.
\end{equation}

{\bf STEP 4}: $L^\infty$ separation.

Finally we prove the $L^\infty$ separation \eqref{thm1_3}. We rewrite \eqref{dhatx} as
\begin{equation}\label{hbxi}\begin{split}
\dot{\hat{\bx}}_i = & \frac{1}{N_1}\sum_{i'\ne i} \psi^+_{i,i'} (\hat{\bx}_{i'}-\hat{\bx}_i) - \frac{1}{N_2}\sum_{j} \psi^-_{i,j} (\hat{\by}_j-\hat{\bx}_i) - \hat{\Psi}^-_i (\bar{\by}-\bar{\bx}) \\
& - \frac{1}{N_1}\sum_{i'} \hat{\Psi}^-_{i'} \hat{\bx}_{i'}+\frac{1}{N_2}\sum_j \hat{\Psi}^-_j \hat{\by}_j \\
= &  - (\Psi^+_i-\Psi^-_i)\hat{\bx}_i + S,  \\
\end{split}\end{equation}
where the source term
\begin{equation}
S := \frac{1}{N_1}\sum_{i'\ne i} \psi^+_{i,i'} \hat{\bx}_{i'} - \frac{1}{N_2}\sum_{j} \psi^-_{i,j} \hat{\by}_j - \hat{\Psi}^-_i (\bar{\by}-\bar{\bx}) - \frac{1}{N_1}\sum_{i'} \hat{\Psi}^-_{i'} \hat{\bx}_{i'}+\frac{1}{N_2}\sum_j \hat{\Psi}^-_j \hat{\by}_j,
\end{equation}
can be estimated by
\begin{equation}\begin{split}
|S|^2 \le & 5\left(\left|\frac{1}{N_1}\sum_{i'\ne i} \psi^+_{i,i'} \hat{\bx}_{i'}\right|^2 + \left| \frac{1}{N_2}\sum_{j} \psi^-_{i,j} \hat{\by}_j\right|^2 + \left| \hat{\Psi}^-_i (\bar{\by}-\bar{\bx})\right|^2 + \left|\frac{1}{N_1}\sum_{i'} \hat{\Psi}^-_{i'} \hat{\bx}_{i'}\right|^2 + \left| \frac{1}{N_2}\sum_j \hat{\Psi}^-_j \hat{\by}_j\right|^2 \right) \\
\le & 5(\var(\bx) + \var(\by) + D(\Psi^-)^2 |\bar{\bx}-\bar{\by}|^2 + \var(\bx) + \var(\by)) \\
\le & 10\left(\var(\bx)+\var(\by) + D(\Psi^-)^2|\bar{\bx}-\bar{\by}|^2\right),
\end{split}\end{equation}
by using the fact that $\psi^+_{i,i'}$ and $\psi^-_{i,j}$ are at most 1. The smallness of $|S|^2$ (compared with $|\bar{\bx}-\bar{\by}|^2$, in view of the $L^2$ estimates \eqref{thm1_2}) enables us to estimate the decay of $|\hat{\bx}_i|^2$. In fact, multiplying \eqref{hbxi} by $\hat{\bx}_i$ and using $\Psi^+_i-\Psi^-_i \ge  \bar{\Psi}^+-\bar{\Psi}^- - D(\Psi^+)-D(\Psi^-)$ gives
\begin{equation}\begin{split}
\frac{1}{2}\frac{\rd}{\rd{t}}|\hat{\bx}_i|^2 \le &  - (\bar{\Psi}^+-\bar{\Psi}^- - D(\Psi^+)-D(\Psi^-))|\hat{\bx}_i|^2 +  |S\cdot\hat{\bx}_i| \\
\le &  - (p-\frac{p}{24}-q-\frac{p}{24}-\frac{p}{12}-\frac{p}{3})|\hat{\bx}_i|^2 +  \frac{3|S|^2}{4p}\\
\le &  - (\frac{p}{2}-q)|\hat{\bx}_i|^2 +  \frac{10}{p}\left(\var(\bx)+\var(\by) + D(\Psi^-)^2|\bar{\bx}-\bar{\by}|^2\right)\\
\le &  - (\frac{p}{2}-q)|\hat{\bx}_i|^2 +  \frac{10}{p}(\lambda_-+\lambda(0)e^{-\mu t} + \frac{1}{N^{1-\alpha}})|\bar{\bx}-\bar{\by}|^2\\
\le &  - (\frac{p}{2}-q)|\hat{\bx}_i|^2 +  C(\lambda(0) e^{-\mu t} + \frac{1}{N^{1-\alpha}})|\bar{\bx}-\bar{\by}|^2,\\
\end{split}\end{equation}
by using \eqref{psi_cond} and \eqref{thm1_2}, and $C$ denotes a constant\footnote{In the rest of this proof, different $C$ may denote different constants which are all independent of $N$.} independent of $N$. By taking $i$ as the index with maximal $|\hat{\bx}_i|^2$ and conducting similar estimate for $|\hat{\by}_j|^2$, we obtain
\begin{equation}\begin{split}
\frac{1}{2}\frac{\rd}{\rd{t}}(\max_i |\hat{\bx}_i|^2 + \max_j |\hat{\by}_j|^2) \le &  - (\frac{p}{2}-q)(\max_i |\hat{\bx}_i|^2 + \max_j |\hat{\by}_j|^2) +  C(\lambda(0)e^{-\mu t} + \frac{1}{N^{1-\alpha}})|\bar{\bx}-\bar{\by}|^2.\\
\end{split}\end{equation}
Integrating in time gives
\begin{equation}\label{infty1}\begin{split}
\max_i |\hat{\bx}_i(t)|^2 + \max_j |\hat{\by}_j(t)|^2 \le &  e^{-(p-2q) t}(\max_i |\hat{\bx}_i(0)|^2 + \max_j |\hat{\by}_j(0)|^2) \\
& +  C\int_0^t e^{-(p-2q)(t-s)}(\lambda(0)e^{-\mu s} + \frac{1}{N^{1-\alpha}})|\bar{\bx}(s)-\bar{\by}(s)|^2\rd{s}.\\
\end{split}\end{equation}

Next we notice that using \eqref{thm1_2} in \eqref{exp1} (with a different choice $c_1=q/24$) gives an exponential growth estimate
\begin{equation}\label{exp1_1}\begin{split}
\frac{\rd}{\rd{t}}|\bar{\bx}-\bar{\by}|^2 \ge & ( 4\bar{\Psi}^- - \frac{q}{12}) |\bar{\bx}-\bar{\by}|^2 - \frac{96D(\Psi^-)^2}{q}(\var(\bx)+\var(\by))\\
\ge & \left( 4q - \frac{q}{6} - \frac{q}{12} - \frac{96D(\Psi^-)^2}{q} (\lambda_-+\lambda(0)e^{-\mu t})\right)|\bar{\bx}-\bar{\by}|^2 \\
\ge & 2q|\bar{\bx}-\bar{\by}|^2,
\end{split}\end{equation}
if $t\ge t_1$, with $t_1$ being independent of $N$ (which follows from $D(\Psi^-)^2\le 1/N^{1-\alpha}$ and $\lambda(0) \le \lambda_+ = O(N^{1-\alpha})$). Therefore
\begin{equation}\label{exp3}
|\bar{\bx}(t)-\bar{\by}(t)|^2 \ge e^{2q(t-s)}|\bar{\bx}(s)-\bar{\by}(s)|^2,\quad \forall t\ge s \ge t_1.
\end{equation}
Using this in \eqref{infty1}, we obtain
\begin{equation}\begin{split}
& \max_i |\hat{\bx}_i(t)|^2 + \max_j |\hat{\by}_j(t)|^2 \\
\le &  C_1e^{-(p-2q) t} +  C\int_{t_1}^t e^{-(p-2q)(t-s)}(\lambda(0)e^{-\mu s} + \frac{1}{N^{1-\alpha}})|\bar{\bx}(s)-\bar{\by}(s)|^2\rd{s}\\
\le &  C_1e^{-(p-2q) t} +  C\int_{t_1}^t e^{-(p-2q)(t-s)}(\lambda(0)e^{-\mu s} + \frac{1}{N^{1-\alpha}})e^{-2q(t-s)}\rd{s} \cdot |\bar{\bx}(t)-\bar{\by}(t)|^2\\
\le &  C_1e^{-(p-2q) t} +  C\left(\lambda(0) e^{-\mu t} +  \frac{1}{N^{1-\alpha}} \right)|\bar{\bx}(t)-\bar{\by}(t)|^2,\\
\end{split}\end{equation}
with 
\begin{equation}\label{C1}
C_1 = \max_i |\hat{\bx}_i(0)|^2 + \max_j |\hat{\by}_j(0)|^2 + C\int_0^{t_1} e^{(p-2q) s}(e^{-\mu s} + \frac{1}{N^{1-\alpha}})|\bar{\bx}(s)-\bar{\by}(s)|^2\rd{s}.
\end{equation}
Therefore
\begin{equation}\label{infty_decay}\begin{split}
\frac{\max_i |\hat{\bx}_i(t)|^2 + \max_j |\hat{\by}_j(t)|^2 }{|\bar{\bx}(t)-\bar{\by}(t)|^2}
\le &  \frac{C_1}{|\bar{\bx}(t)-\bar{\by}(t)|^2}e^{-(p-2q) t} +  C\left(\lambda(0) e^{-\mu t} +  \frac{1}{N^{1-\alpha}} \right).\\
\end{split}\end{equation}
Finally we notice that the second inequality in \eqref{exp1_1} implies that 
\begin{equation}
\frac{\rd}{\rd{t}}|\bar{\bx}-\bar{\by}|^2 \ge -C_2|\bar{\bx}-\bar{\by}|^2,
\end{equation}
for $0\le t\le t_1$, with $C_2$ independent of $N$. Therefore, for $0\le s \le t_1$,
\begin{equation}
|\bar{\bx}(t_1)-\bar{\by}(t_1)|^2 \ge e^{-C_2(t_1-s)}|\bar{\bx}(s)-\bar{\by}(s)|^2,
\end{equation}
which implies
\begin{equation}
|\bar{\bx}(s)-\bar{\by}(s)|^2 \le e^{C_2t_1}|\bar{\bx}(t_1)-\bar{\by}(t_1)|^2 \le e^{C_2t_1}e^{-2q(t-t_1)}|\bar{\bx}(t)-\bar{\by}(t)|^2,
\end{equation}
where we used \eqref{exp3} in the last inequality. It follows that
\begin{equation}\begin{split}
\frac{C_1}{|\bar{\bx}(t)-\bar{\by}(t)|^2}e^{2q t} = & \frac{\max_i |\hat{\bx}_i(0)|^2 + \max_j |\hat{\by}_j(0)|^2}{|\bar{\bx}(0)-\bar{\by}(0)|^2}\cdot \frac{|\bar{\bx}(0)-\bar{\by}(0)|^2}{|\bar{\bx}(t)-\bar{\by}(t)|^2}e^{2qt} \\
& + C\int_0^{t_1} e^{(p-2q) s}(e^{-\mu s} + \frac{1}{N^{1-\alpha}})\frac{|\bar{\bx}(s)-\bar{\by}(s)|^2}{|\bar{\bx}(t)-\bar{\by}(t)|^2}\rd{s}\cdot e^{2qt} \\
\le &  e^{C_2t_1+2qt_1}\tilde{\lambda}(0) + Ce^{C_2t_1+2qt_1+pt_1} \le C(1+\tilde{\lambda}(0)).
\end{split}\end{equation}
Therefore the first term on RHS of \eqref{infty_decay} behaves like $C(1+\tilde{\lambda}(0))e^{-pt}$ which decays faster than $e^{-\mu t}$. Therefore \eqref{infty_decay} gives the $L^\infty$ separation \eqref{thm1_3} since $\lambda(0)\le \tilde{\lambda}(0)$.

\end{proof}

\begin{remark}\label{rem_proof1}
In fact, \eqref{psi_cond} are the only conditions on the coefficient matrices used in this proof. Therefore, the proof applies to \eqref{eq} (not necessarily with \eqref{rand1}) on $[0,T]$, as long as the coefficient matrices satisfies \eqref{psi_cond} for any $t\in [0,T]$. This will be the starting point of the proof of Theorem \ref{thm2} given below.
\end{remark}

\begin{proof}[Proof of Theorem \ref{thm2}]

Recall that Theorem \ref{thm2} assumes the coefficient matrices are given by the random piecewise constants \eqref{rand2}. We will fix the choice of $\alpha\in (0,1)$. Fix $K>0$ to be chosen, it is clear that with large probability, \eqref{psi_cond} holds for $[0,T],\,T=K\tau$. In fact, for any $\epsilon>0$, for any fixed $k$, if $N\ge N_0(\epsilon,p,q)$, then (by Lemmas \ref{lem_psi1} and \ref{lem_psi2}) with probability $1-\epsilon$, \eqref{psi_cond} holds for $[k\tau,(k+1)\tau]$. By the independence of the coefficient matrices for different $k$,
\begin{equation}
P\Big(\text{\eqref{psi_cond} holds for $[0,K\tau]$}\Big) \ge (1-\epsilon)^K \rightarrow 1,\quad \text{ as } \epsilon\rightarrow 0,
\end{equation}
which gives \eqref{psi_cond} for $t\in [0,T]$ with large probability. 

The assumption \eqref{thm1_1} on the initial data also holds for large $N$, since $\lambda(0)\le \tilde{\lambda}(0) \le \Lambda \le \lambda_+(N)$. Therefore, by Remark \ref{rem_proof1}, the estimate \eqref{thm1_3} holds for \eqref{eq} with \eqref{rand2} with large probability, on the fixed time interval $[0,T]$. Since $\lambda_-\sim 1/N^{1-\alpha}$ and $\tilde{\lambda}(0)\le \Lambda$, by taking $N$ and $K$ large enough (in terms of $(\Lambda,\tau,p,q,\kappa)$, with $K$ being independent of $N$), one can guarantee that the RHS of \eqref{thm1_3} at time $T=K\tau$ is no more than $1/16$. Therefore, with large probability, if $N$ is sufficiently large, then
\begin{equation}\label{lambdaT}
\tilde{\lambda}(T) \le \frac{1}{16},\quad T = K\tau,\quad K = K(\Lambda,\tau,p,q,\kappa).
\end{equation}
Notice that
\begin{equation}
|\hat{\bx}_i\cdot(\bar{\by}-\bar{\bx})| \le |\hat{\bx}_i|\cdot|\bar{\by}-\bar{\bx}| \le \sqrt{\tilde{\lambda}}|\bar{\by}-\bar{\bx}|^2,
\end{equation}
and similar for $\by$. Thus by \eqref{lambdaT}, for any $i$ and $j$,
\begin{equation}\begin{split}
(\by_j(T)-\bx_i(T))\cdot \frac{\bar{\by}(T)-\bar{\bx}(T)}{|\bar{\by}(T)-\bar{\bx}(T)|^2} = & \hat{\by}_j(T)\cdot \frac{\bar{\by}(T)-\bar{\bx}(T)}{|\bar{\by}(T)-\bar{\bx}(T)|^2} - \hat{\bx}_i(T)\cdot \frac{\bar{\by}(T)-\bar{\bx}(T)}{|\bar{\by}(T)-\bar{\bx}(T)|^2} + 1 \\
\ge 1 - \frac{1}{4} - \frac{1}{4} = \frac{1}{2} > 0.
\end{split}\end{equation} 
Therefore \eqref{thm2_2} holds at time $T$, with the vector\footnote{This means that at time $T$ the two groups are separated by the perpendicular bisector of the segment connecting $\bar{\bx}(T)$ and $\bar{\by}(T)$.} $\bv=\frac{\bar{\by}(T)-\bar{\bx}(T)}{|\bar{\by}(T)-\bar{\bx}(T)|}$.

Finally we show that \eqref{thm2_2} at time $T$ implies \eqref{thm2_2} for all time $t\ge T$. In fact, \eqref{thm2_2} at time $T$ implies the existence of a constant $c$ such that
\begin{equation}
\bx_i(T)\cdot \bv < c,\quad \by_j(T)\cdot \bv > c,\quad \forall i,\,j.
\end{equation}
Taking dot product of \eqref{eq} with $\bv$ gives
\begin{equation}\begin{split}
& \dot{x}_i = \frac{1}{N_1}\sum_{i'\ne i} \psi^+_{i,i'}(t) (x_{i'}-x_i) + \frac{1}{N_2}\sum_{j} \psi^-_{i,j}(t) (y_j+x_i), \\
& \dot{y}_j = \frac{1}{N_2}\sum_{j'\ne j} \psi^+_{j,j'}(t) (y_{j'}-y_j) + \frac{1}{N_1}\sum_{i} \psi^-_{j,i}(t) (x_i+y_j), \\
\end{split}\end{equation}
where $x_i:=-(\bx_i\cdot\bv - c)$, $y_j:=\by_j\cdot\bv - c$ satisfy $x_i(T)>0,\,y_j(T)>0$. This ODE system clearly propagates the positiveness of $\{x_i\}$ and $\{y_j\}$ if all the coefficients $\psi^+_{i,i'},\psi^-_{i,j},\psi^+_{j,j'}$ are nonnegative. Therefore
\begin{equation}
\bx_i(t)\cdot \bv < c,\quad \by_j(t)\cdot \bv > c,\quad \forall i,\,j,
\end{equation}
for all $t\ge T$, and \eqref{thm2_2} follows.
 
\end{proof}

\section{Numerical results}

In this section we provide a few numerical results verifying Theorems \ref{thm1} and \ref{thm2}. 

For both models (\eqref{eq} with \eqref{rand1} and \eqref{eq} with \eqref{rand2}), we take the spatial dimension to be one, $N_1=N_2=:N$, and
\begin{equation}
p = 0.3,\quad q = 0.2,
\end{equation}
while for  \eqref{eq} with \eqref{rand2}, we take the time step $\tau=1$ for changing communication matrix. We take the initial data $\{x_i(0)\}$ and $\{y_j(0)\}$ to be i.i.d. uniformly distributed random variables in $[0,1]$, and compute the solution at $T=20$ to \eqref{eq} exactly by using matrix exponentials. Figure~\ref{fig1} shows one simulation for each model with $N=40$. For the sake of demonstration\footnote{Without this normalization, $x_i(t)$ and $y_j(t)$ may exhibit exponential growth in time, which makes it hard to demonstrate the extent of separation.}, at each $t$ we normalize the vector $(x_1(t),\dots,x_N(t),y_1(t),\dots,y_N(t))$ to have $\ell^2$ norm equal to $\sqrt{2N}$, and this does not change the dynamics since \eqref{eq} is linear. 

\begin{figure}
\begin{center}
	\includegraphics[width=0.49\textwidth]{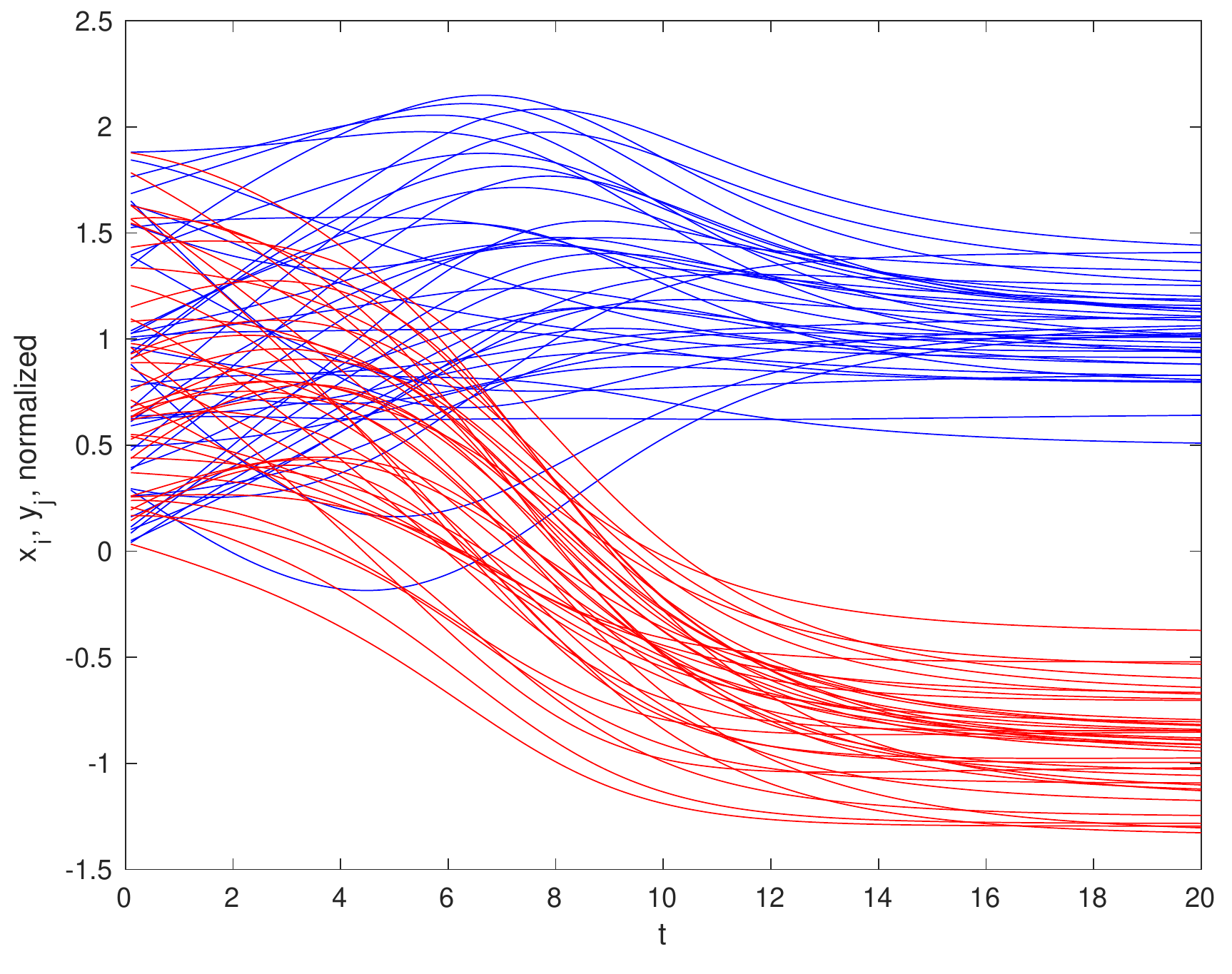}
	\includegraphics[width=0.49\textwidth]{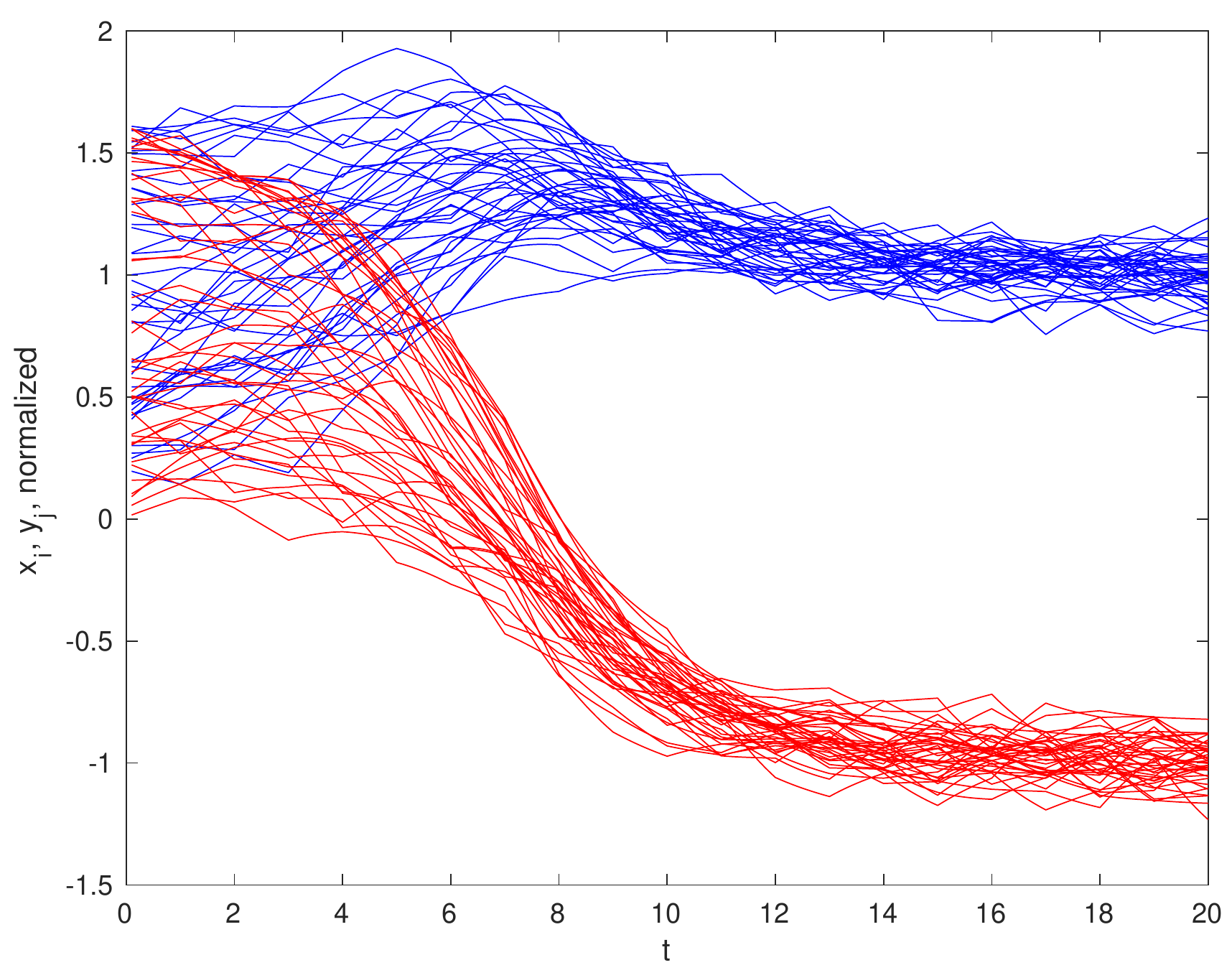}
	\caption{Simulation for \eqref{eq}, with $N_1=N_2=40$. Left: with \eqref{rand1}; right: with \eqref{rand2}. The horizontal axis is time, and the vertical axis is $x_i(t)$ and $v_j(t)$, normalized. Red and blue curves represent the individuals in the first and second groups, respectively.}
	\label{fig1}
\end{center}	
\end{figure}

One can see that for both models, the two groups completely separate from each other. Also, after getting to a certain extent of separation, $\var(x)$ and $\var(y)$ no longer shrink, compared to $|\bar{x}-\bar{y}|^2$. This is exactly as predicted by Theorem \ref{thm1}: the separation indicator $\lambda(t)$, as defined in \eqref{thm1_2}, decays to some positive constant $\lambda_-$ but does not necessarily converge to zero.

To further investigate the limiting behavior of $\lambda(t)$ as $t$ getting large for \eqref{eq} with \eqref{rand1}, we compute this model with randomly sampled initial data and coefficient matrices for $n_{test}=10000$ times, for various values of $N$. The sample average of $\lambda(T)$ is computed by discarding $n_{discard}=100$ samples with largest $\lambda(T)$ which may contain the cases with small probability for which the separation claimed in Theorem \ref{thm1} may fail. The relation of $N$ and the sample average of $\lambda(T)$ is plotted in Figure \ref{fig2}. It is clear that the latter behaves like $O(1/N)$ for large $N$, which means that the estimate $\lambda_- \le C/N^{1-\alpha},\, \forall 0<\alpha<1$, given in \eqref{lambda}, is optimal.

\begin{figure}
\begin{center}
	\includegraphics[width=0.70\textwidth]{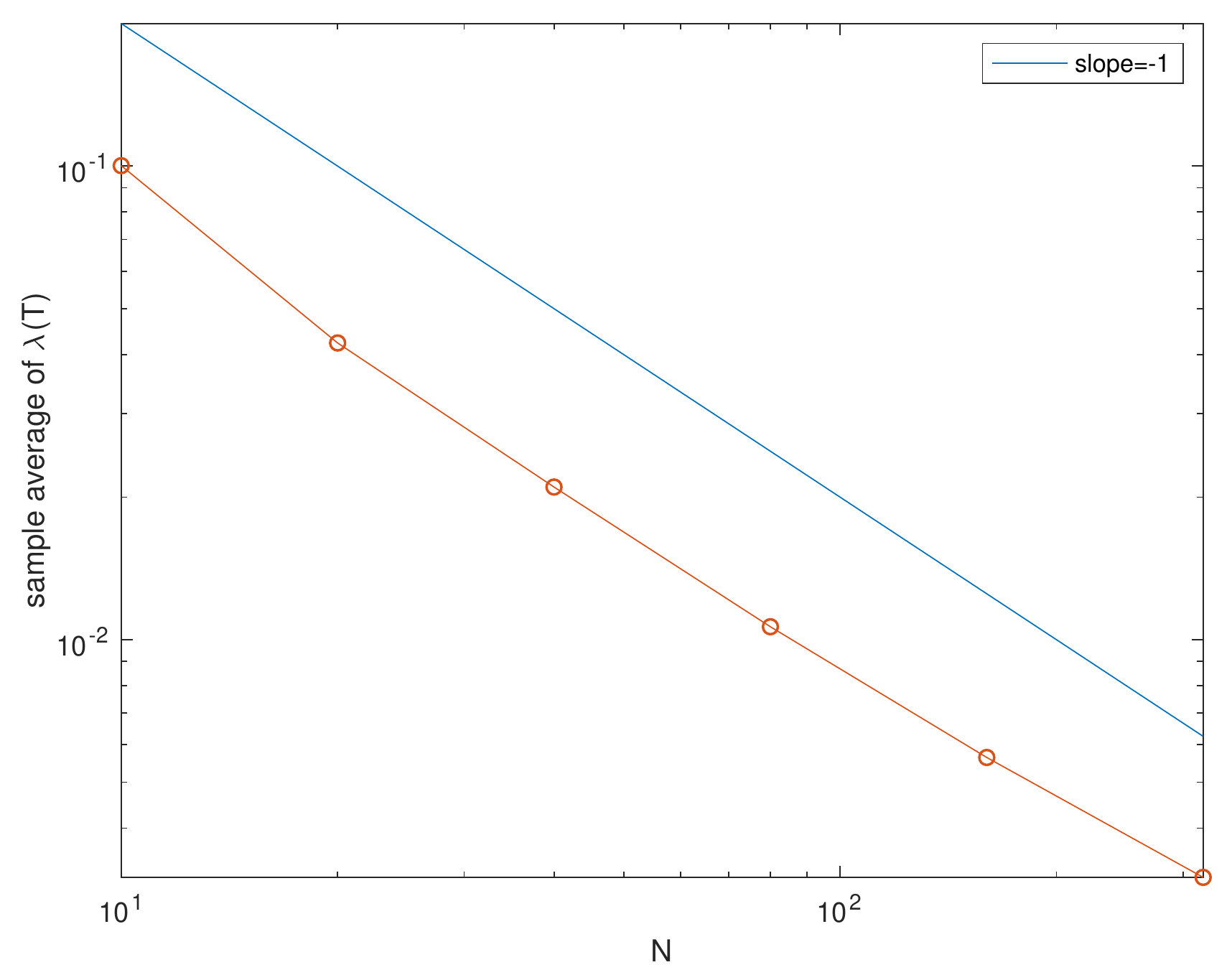}
	\caption{Sample average of $\lambda(T)$, $T=20$, for \eqref{eq} with \eqref{rand1}, for various $N$. $n_{test}=10000$ samples are taken, while $n_{discard}=100$ samples with largest $\lambda(T)$ are discarded. The horizontal axis is $N$, the circles are sample averages of $\lambda(T)$, and the straight line is slope -1.}
	\label{fig2}
\end{center}	
\end{figure}

\section{Conclusion}

In this paper we propose collective dynamical models describing the interaction of two opposing groups with stochastic communication, where inter-group and intra-group communication functions have opposite signs. Under reasonable assumptions we prove the {\it separation} of opinions of the two groups. Compared with existing results for collective dynamics, our results suggest:
\begin{itemize}
\item Anti-alignment between opposing groups and alignment within the same group may lead to separation of two clusters, instead of consensus. This may happen even if the initial data is {\it well-mixed}.
\item With stochastic communication, the expected large time behavior may still be observed with large probability.
\end{itemize}
There are a few future topics to study:
\begin{itemize}
\item Generalizing our results to cases where the coefficients depend on the locations $\bx_i$ and $\by_j$.
\item Generalizing our results to models with more than two groups. Numerical examples in~\cite{JLL} have shown the separation for three groups.
\item Second-order Cucker-Smale models~\cite{CS1,CS2,HT} with anti-alignment and possibly stochastic communication. Some partial results in this direction are already obtained by~\cite{FHJ} in deterministic setting.
\end{itemize}

\bibliographystyle{plain}
\bibliography{twogroup_separation_bib}

\end{document}